\documentclass[12pt]{article}%
\usepackage{amsmath}
\usepackage{amsfonts}
\usepackage{amssymb}
\usepackage{graphicx}%
\providecommand{\U}[1]{\protect\rule{.1in}{.1in}}
\newtheorem{theorem}{Theorem}

\newtheorem{definition}[theorem]{Definition}

\newtheorem{lemma}[theorem]{Lemma}

\newtheorem{remark}[theorem]{Remark}

\newenvironment{proof}[1][Proof]{\noindent\textbf{#1.} }{\ \rule{0.5em}{0.5em}}
\begin{document}

\title{Some characterizations of ergodicity in Riesz spaces}
\author{Youssef Azouzi\thanks{The authors are members of the GOSAEF research group} ,
Marwa Masmoudi.
\and {\small Research Laboratory of Algebra, Topology, Arithmetic, and Order}\\{\small Department of Mathematics}\\COSAEF {\small Faculty of Mathematical, Physical and Natural Sciences of
Tunis}\\{\small Tunis-El Manar University, 2092-El Manar, Tunisia}}
\maketitle

\begin{abstract}
In the recent surge of papers on ergodic theory within Riesz spaces, this
article contributes by introducing enhanced characterizations of ergodicity.
Our work extends and strengthens prior results from both the authors and
Homann, Kuo, and Watson. Specifically, we show that in a conditional
expectation preserving system $(E,T,S,e),$ $S$ can be extended to $L%
%TCIMACRO{\U{b9}}%
%BeginExpansion
{{}^1}%
%EndExpansion
(T)$ and operates as an isometry on $L^{p}\left(  T\right)  $ spaces.

\end{abstract}

\section{Introduction}

The first paper in ergodic theory in the setting of Riesz spaces has been
developed by Kuo, Labuschagne, and Watson \cite{L-33}. Many important concepts
and results in ergodic theory have been generalized to the Riesz space setting
\cite{L-950,L-724,L-33,L-929,L-797}. Ergodic systems are those for them time
means and spaces means are equal. This definition of ergodicity comes back to
Boltzmann. This concept has many applications in different fields such as
physics in the understanding of the behavior of particles in a gas, as
economics when analyzing the long-term behavior of financial markets
[\cite{a-2404}, \cite{a-2405}], information theory, particularly in the study
of stochastic processes and data compression \cite{a-2403}. A question of
interest is to find conditions on a transformation that imply ergodicity. In
the classical case of probability spaces, several alternative
characterizations of ergodicity were given \cite{b-1662,b-2982,a-1051,b-157}.
In the Riesz spaces setting, the theory has been started in 2007 in
\cite{L-33} where the authors extended Birkhoff, Hopf, and Wiener ergodic
theorems without giving an explicit definition of ergodicity. It was only in
2021 that a concrete definition was introduced (see \cite{L-724}). We recall
from this paper that a 4-uplet $\left(  E,T,e,S\right)  $ is called a
conditional expectation preserving system if $\left(  E,T,e\right)  $ is a
conditional Riesz triple and $S$ is a lattice homomorphism satisfying $TS=T$
and $Se=e.$ The system attains ergodicity when, for every $S$-invariant
element $f$ in $E$, the Cesaro mean $\dfrac{1}{n}%
%TCIMACRO{\tsum \limits_{k=0}^{n-1}}%
%BeginExpansion
{\textstyle\sum\limits_{k=0}^{n-1}}
%EndExpansion
S^{k}f$ converges in order to $Tf$. This faithfully reflects the principle
that time and space means align within classical theory. We will show in this
work that this result can be extended to elements in $L^{1}\left(  T\right)  $
and not only to $S$-invariant elements. This expansion aligns seamlessly with
classical theory. In our recent paper co-authored with Ben Amor, Homann and
Watson \cite{L-950} we employed an alternate equivalent definition of
ergodicity rooted in classical theory stating that $S$-invariant elements are
also $T$-invariant. We will refine this result by showing that validating this
condition for components is sufficient. Let us give a brief overview of the
paper's content.

In Section 2 we provide some preliminaries that will be needed in the
following sections. We view ergodicity within Riesz spaces, highlighting its
correlation with classical cases. Then we provide a new equivalent definition
of ergodicity involving only components. A remarkable significant result
stating that if $\left(  E,T,S,e\right)  $ is a conditional expectation
preservingsystem, then $ST=T$ has not been observed before. Moving to Section
3, we articulate a well-known probabilistic result characterizing ergodicity.
We elucidate the translation from the classical setting
\cite{b-157,b-938,b-3102} to the framework of Riesz spaces. This transition
involves interpreting measurable subset measures as images of components
through the conditional expectation operator $T$. Section 4 focuses on
extending $S$ to the whole of $L^{1}\left(  T\right)  .$ This extension
broadens the scope of the equality of time and space means for $S$-invariant
elements, a result initially established in \cite{L-724}, now extended to the
domain of $T$ denoted by $L^{1}(T)$. In the last section we offer some
applications. In particular we prove that $S$ acts as an isometry from
$L^{p}\left(  T\right)  $ to itself via the $R\left(  T\right)  $-valued norm
$\left\Vert .\right\Vert _{T;p}.$ Then we get several characterizations of
ergodicity that strenghten a previous result obtained in \cite{L-724}. For
more knowledge about ergodicity in the classical case we refer the reader to
\cite{b-3102,b-157,b-1662}. An operator theory approach that is more suitable
for our study is presented in \cite{b-2982}. For Riesz spaces theory that is
not recalled here the reader can consult \cite{b-1089}.

\section{Preliminaries}

A triple $(E,e,T)$ is called a conditional Riesz triple if $E$ is a Dedekind
complete Riesz space, $e$ is an order weak unit of $E$ and $T$ is a
conditional expectation operator on $E$, that is, a strictly positive order
continuous linear projection with $Te=e$ and having a Dedekind complete range.

In \cite{L-724}, the authors introduced the concept of a conditional
expectation preserving system, which is defined as a $4$-tuple $(E,T,S,e)$,
where $(E,e,T)$ is a Riesz triple and $S$ is an order continuous Riesz
homomorphism on $E$ satisfying $Se=e$ and $TSf=Tf$ for all $f\in E$. For $f\in
E$ and $n\in\mathbb{N}$ we set%
\[
S_{n}f:=\frac{1}{n}\sum_{k=0}^{n-1}S^{k}f,
\]
and $L_{S}$ the order limit of this sequence when this limit exists.

The set of $S$-invariant elements of $E$ will be denoted by $\mathcal{I}_{S}$,
that is,%
\[
\mathcal{I}_{S}:=\{f\in E:Sf\ =f\}.
\]

In the classical case of a measure preserving system

$(\Omega,\Sigma,\mu,\tau)$, a set $B\in\Sigma$ is called $\tau$-invariant if
$\tau^{-1}B=B$ and the system $(\Omega,\Sigma,\mu,\tau)$ is said to be ergodic
if every $\tau$-invariant set has measure $0$ or $1$, which is equivalent to
$\chi_{B}$ being the zero or one constant function in $L^{1}(\Omega,\Sigma
,\mu)$ sense. If we define $Sf:=f\circ\tau$, then we can not adopt this
classical definition in the setting of Riesz spaces by assuming that the only
$S$-invariant components are zero and the weak order unit $e$. In fact, by the
Freudenthal theorem, this assumption leads to a trivial result where the
dimension of the range of $T$ is equal to $1$. In other words, we end up with
a situation where $T$ is a functional, which is equivalent to the classical
case. This approach is not very interesting as it does not provide any new
insights. This idea is similar to the problem of extending $L^{\infty}$ spaces
to the setting of Riesz spaces. In the classical theory, for a probability
space $(\Omega,\Sigma,\mu)$,

\begin{center}
$L^{\infty}(\mu)=\{f\in L^{1}(\mu):|f|\leq\lambda$ for some $\lambda
\in\mathbb{R}\}$.
\end{center}

In the setting of Riesz spaces, Trabelsi and Azouzi in \cite{L-180} defined,
for a conditional expectation $T$, the space $L^{\infty}(T)$ as follows:

\begin{center}%
\[
L^{\infty}(T)=\{f\in L^{1}(T):f\leq u\text{ for some }u\in R(T)\}.
\]

\end{center}

A measure preserving transformation $\tau$ induces an operator

$S:L^{1}\longrightarrow L^{1}$ which maps $f$ to $f\circ\tau.$ Ergodicity can
be expressed as follows: The system $L^{1}\left(  \Omega,\Sigma,\mu
,\tau\right)  $ is ergodic if $Sf=f$ implies $f$ is constant almost
everywhere. Or, equivalently, $Sf=f$ implies $f\in R\left(  \mathbb{E}\right)
,$ where $\mathbb{E}$ is the expectation operator. The definition below taken
from \cite{L-950} provides a natural way of extending the concept of
ergodicity to the context of Riesz spaces.

\begin{definition}
\textit{The conditional expectation preserving system} $(E,T,S,e)$ \textit{is
said to be ergodic if}%
\[
Sf=f\ \text{\textit{imples}}\ Tf=f.
\]

\end{definition}

An alternative definition of ergodicity has been proposed earlier by Homann,
Kuo, and Watson in \cite{L-724} stating that a conditional expectation
preserving system $(E,T,S,e)$ is ergodic if it satisfies the condition

\begin{center}
$L_{S}f\in R(T)$ for all $f\in\mathcal{I}_{S}$.
\end{center}

However, the condition $f\in\mathcal{I}_{S}$ implies immediately that
$L_{S}f=f$ and so $f=L_{S}f\in R\left(  T\right)  .$ Noting this, Theorem 3.2
in \cite{L-724} becomes trivial and it is just a rephrasing of the definition
provided just before \cite[Definition 3.1]{L-724}.

Events in probability theory correspond to components in Riesz space theory.
The following lemma generalizes a basic fact about ergodicity by expressing it
in terms of components. It will be used several times in our work.

\begin{lemma}
\label{LC}\textit{The conditional expectation preserving system} $(E,T,S,e)$
\textit{is ergodic if, and only if}%
\[
Sp=p\ \text{implies}\ Tp=p\ \text{for\ all component }p\text{ of }e.
\]

\end{lemma}

\begin{proof}
We need only to prove that the condition is sufficient. Consider $f\in E^{+}$
such that $Sf=f.$ Then $f$ is an order limit of a sequence $\left(
f_{n}\right)  $ where $f_{n}$ is a linear combination of components of the
form$\ p_{\alpha}=P_{\left(  \alpha e-f\right)  ^{+}}e,$ $\alpha\in
\mathbb{R}.$ We claim that $Sp_{\alpha}=p_{\alpha}$. Indeed $p_{\alpha}%
=\sup\limits_{n}\left(  e\wedge n\left(  \alpha e-f\right)  ^{+}\right)  $ and
as $S$ is order continuous and lattice homomorphism we get%
\[
Sp_{\alpha}=\sup\limits_{n}\left(  Se\wedge n\left(  \alpha Se-Sf\right)
^{+}\right)  =\sup\limits_{n}\left(  e\wedge n\left(  \alpha e-f\right)
^{+}\right)  =p_{\alpha},
\]
which proves the claim. It follows from our assumption that $Tp_{\alpha
}=p_{\alpha}$ for all real $\alpha$ and then $Tf=f$ by linearity and order
continuity of $T.$
\end{proof}

The main idea behind extending ergodic theory to the setting of Riesz spaces
is to formulate everything in terms of operators. Two pivotal operators come
into play: the condition expectation operator $T$ and the Riesz homomorphism
$S.$ These operators exhibit commutativity and adhere to the relationship
$ST=TS=T.$ Only condition $TS=T,$ which is imposed by definition, was explored
in the previous papers dealing with the subject. Condition $ST=T,$ however,
has not been remarked before and will be proved in our forthcoming lemma.

\begin{lemma}
\label{ST}\textit{Let} $(E,T,S,e)$ \textit{be a conditional expectation
preserving system. Then }$Sf=f\ $for\ all$\ f\in R(T).$
\end{lemma}

\begin{proof}
By Freudenthal Theorem \cite[Theorem 33.3]{b-1089}, every vector in $R\left(
T\right)  $ is an order limit of a sequence of linear combination of
components. As $S$ is linear and order continuous it is sufficient to prove
that $S$ and $\operatorname*{Id}\nolimits_{R\left(  T\right)  }$ coincide on
components of $R\left(  T\right)  .$ Pick then a component $f$ in the range of
$T$. Then%
\[
Sf\wedge\left(  e-Sf\right)  =Sf\wedge S\left(  e-f\right)  =S\left(
f\wedge\left(  e-f\right)  \right)  =0.
\]
Hence $Sf$ is a component of $e.$

On the other hand, the averaging property of $T$ yields%
\[
T\left(  fSf\right)  =fTSf=fTf=f^{2}=f=Tf.
\]
Now since $T$ is strictly positive and $f\geq fSF$ we get $f=fSf.$. This
implies that $f\leq Sf$ and using again the strict positivity of $T$ gives
$Sf=f.$ This completes the proof.
\end{proof}

We record here a simple but useful fact.

\begin{lemma}
\label{TpC}Let $p$ be a component of $e$ such that $Tp$ is a component of $e$.
Then $Tp=p.$
\end{lemma}

\begin{proof}
By the averaging property of $T$ we have%
\[
Tp=Tp.Tp\ =T\ (p.Tp),
\]
which yields $T(p-pTp)=0$. But since $p-pTp\geq0$ and $T$ is strictly positive
we deduce that $p=pTp=p\wedge Tp$. Hence $p\leq Tp,$ and again the strict
positivity of $T$ implies $p=Tp$.
\end{proof}

\section{Ergodicity via conditional expectation}

Recall that a dynamical system is said to be ergodic if it has no nontrivial
invariant subsets. Numerous alternative characterizations of ergodicity can be
found in the literature (see \cite{b-157,b-938,b-3102}.

\begin{theorem}
\label{T3}\textit{Let} $(\Omega,\Sigma,\mu,\tau)$ \textit{be a measure
preserving system. The following statements are equivalent}.

\begin{enumerate}
\item[(i)] $(\Omega,\Sigma,\mu,\tau)$ \textit{is ergodic};

\item[(ii)] \textit{For any} $B\in\Sigma\ $with $\mu(\tau^{-1}(B)\triangle
B)=0$ we have $\mu(B)\in\{0,1\}$;

\item[(iii)] \textit{For any} $A\in\Sigma\ $with $\mu(A)>0$ we have
$\mu(\bigcup_{n=1}^{\infty}\tau^{-n}A)=1$.
\end{enumerate}
\end{theorem}

Before presenting and proving a Riesz space version of Theorem \ref{T3} let us
give some commentary on its statement.

\begin{remark}
\textit{Poincar\'{e} Recurrence Theorem states that for any measure preserving
system }$(\Omega,\Sigma,\mu,\tau)$ \textit{and any set} $A$ \textit{of
positive measure, almost every point of} $A$ \textit{is recurrent with respect
to }$A$\textit{. Ergodicity strengthens this statement by asserting that every
point in the entire space} $\Omega$ \textit{is recurrent with respect to} $A$,
\textit{as described in} (iii).
\end{remark}

We state and prove now an analogous of Theorem \ref{T3}.

\begin{theorem}
\label{T4}\textit{Let} $(E,T,S,e)$ \textit{be a conditional expectation
preserving system. The following statements are equivalent}.

\begin{enumerate}
\item[(i)] The system $\left(  E,T,S,e\right)  $ \textit{is ergodic};

\item[(ii)] \textit{For any component} $p$ \textit{of} $e,\ T((e-p)Sp)=0$
\textit{implies that} $p\in R(T);$

\item[(iii)] \textit{For any component} $p$ \textit{of} $e$ we have $%
%TCIMACRO{\tbigvee \limits_{n=1}^{\infty}}%
%BeginExpansion
{\textstyle\bigvee\limits_{n=1}^{\infty}}
%EndExpansion
S^{n}p\in R(T).$
\end{enumerate}
\end{theorem}

\begin{remark}

\begin{enumerate}
\item \textit{As previously mentioned, the assumptions in the Riesz space
setting differ from those in the classical framework. For instance, the
condition} (\textit{ii}) \textit{in Theorem \ref{T3}}, \textit{can be written
as follows}%
\[
\mathbb{E}\left[  1_{B}1_{(\tau^{-1}B)^{c}}+1_{B^{C}}.1_{\tau^{-1}\left(
B\right)  }\right]  =0.
\]
\textit{Its analogous in our setting should be}%
\[
T(p(e-Sp))+T((e-p)Sp)\ =0.
\]
\textit{We note that our condition suggested in Theorem \ref{T4} is weaker}.

\item \textit{It can be observed that the proofs in the setting of Riesz
spaces are simpler and more straightforward than those in the classical
setting}.
\end{enumerate}
\end{remark}

\begin{proof}
(\textit{i}) $\Rightarrow$ (\textit{ii}). Suppose that the system is ergodic
and let $p$ be a component of $e$ with $T(Sp(e-p))=0$. Since $T$ is strictly
positive and $pSp\leq p$, it follows that $pSp=p,$ which implies that $p\leq
Sp$. But $Tp=TSp$. Again from the fact that $T$ is strictly positive, we get
$p=Sp$. Hence, by ergodicity, $p\in R(T)$.

(ii) $\Rightarrow$ (\textit{iii}). Let $p$ be a component of $e$ and $c=%
%TCIMACRO{\tbigvee \limits_{n=1}^{\infty}}%
%BeginExpansion
{\textstyle\bigvee\limits_{n=1}^{\infty}}
%EndExpansion
S^{n}p$. Since for all $n\in\mathbb{N},\ S^{n}p$ is a component of $e$ and the
set of components of the weak order unit is a complete Boolean algebra (add
reference), we obtain that $c$ is a component of $e$. As $S$ is a Riesz
homomorphism, we have $Sc\leq c$. Using the fact that $TS=T$ and that $T$ is
strictly positive, we obtain that $c$ is an $S$-invariant component of $e$.
Since%
\[
0=T(c(e-c))=T(Sc(e-c)),
\]
by (\textit{ii}) we obtain $c\in R(T)$.

(\textit{iii}) $\Rightarrow$ (\textit{i}) According to Lemma \ref{LC} it is
sufficient to prove that $Tp=p$ for each component $p$ satisfying $Sp=p$. Let
$p$ be such component, then (\textit{iii}) gives $p=%
%TCIMACRO{\tbigvee \limits_{n=1}^{\infty}}%
%BeginExpansion
{\textstyle\bigvee\limits_{n=1}^{\infty}}
%EndExpansion
S^{n}p\in R(T).$ Thus the system is ergodic.
\end{proof}

\section{Extension of $S$}

In \cite{b-1662}, it was shown that a dynamical system $(\Omega,\Sigma
,\mu,\tau)$ is ergodic if and only if for each $f\in L^{1}(\Omega,\Sigma,\mu
)$, the time mean of $f$ equals the space mean of $f$ a.e. That is,%
\[
\lim_{n\rightarrow\infty}\frac{1}{n}\sum_{k=0}^{n-1}f(\tau^{k}x)=\int_{\Omega
}fd\mu\ \text{a.e.}%
\]

In \cite{L-724}, the authors extended this result to the setting of Riesz
spaces. They established that a conditional expectation preserving system
$(E,T,S,e)$ is ergodic if and only if

\begin{center}
$\lim\limits_{n\rightarrow\infty}\dfrac{1}{n}\sum\limits_{k=0}^{n}S^{k}f=Tf$
for all $f\in\mathcal{I}_{S}$.
\end{center}

It is worth noting that the characterization provided above does not precisely
align with the classical definition of ergodicity where a function $f$ is
required to be in $L^{1}(\Omega,\Sigma,\mu)$, instead of just being $\tau
$-invariant. We will show below this result by requiring that $f$ belongs to
$L^{1}(T)$ which is the analogous space of $L^{1}(\Omega,\Sigma,\mu)$ in the
setting of Riesz spaces. In addition, the existence of $L_{S}$, as shown in
\cite{L-33}, requires that we work in a $T$-universally complete space. This
is why the space $L^{1}(T)$ is a good choice. It was shown in \cite{L-24} that
the conditional expectation $T$ can be extended to its natural domain
$L^{1}(T)$. Before stating our result about ergodicity we need to prove first
that $S$ also can be extended to $L^{1}(T)$.

Recall from \cite{L-24} that $L^{1}(T)=D-D$ \textit{where}

\begin{center}
$D=$ \{$x\in E_{+}^{u}:\exists(x_{\alpha})\subset E_{+},\ 0\leq x_{\alpha
}\uparrow x,\ (Tx_{\alpha})$ order bounded\}.
\end{center}

An equivalent definition has been suggested by Grobler as follows:

$D=\left\{  x\in E_{+}^{s}:\sup Tx_{\alpha}\in E^{u}\text{ for some net
}\left(  x_{\alpha}\right)  \ \text{ in }E_{+}\text{ with }x_{\alpha}\uparrow
x\right\}  $. The discussion concerning the equivalence of definitions
obtained through both the universal completion and the sup-completion
approaches has been detailed in \cite{L-444}. In the proof of Theorem
\ref{Ext} below, we opt for the sup-completion method as it allows for a more
concise argument.

We know by \cite[Proposition 3]{L-444} that $S$ has a unique extension
$S^{\ast}:E^{s}\longrightarrow E^{s}$ which is increasing, left continuous.
Moreover it is additive and positively homogenous. \textit{For every} $x\in
E^{s},$ $S^{\ast}x=\sup\limits_{\alpha}Sx_{\alpha}.$\textit{where}
$(x_{\alpha})$ is any net in $E_{+}$ satisfying $x_{\alpha}\uparrow x$.

\begin{theorem}
\label{Ext}\textit{Let} $(E,T,S,e)$ \textit{be a conditional expectation
preserving system. Then }$S$ has \textit{a unique order continuous Riesz
homomorphism extension to} $L^{1}(T),$ denoted again by $S,$ which fulfills
$TS=T.$
\end{theorem}

\begin{proof}
Let us consider the extensions $S^{\ast}$ defined above. A natural way to
extend $S^{\ast}$ to all of $L^{1}\left(  T\right)  $ involves defining%
\[
S^{\ast}\left(  f\right)  :=S^{\ast}(f^{+})-S^{\ast}(f^{-}),\qquad\text{for
}f\in L^{1}(T).
\]
It follows from \cite[Proposition 3]{L-444} that $S^{\ast}$ is order
continuous on $L^{1}\left(  T\right)  $ once we prove that $S^{\ast}$ maps
$L^{1}\left(  T\right)  $ into itself. To this end let $x\in L^{1}\left(
T\right)  _{+}$ and $(x_{\alpha})$ be a net in $E_{+}\ $with $x_{\alpha
}\uparrow x.$ Then $Sx_{\alpha}\uparrow S^{\ast}x$ and then $TSx_{\alpha
}\uparrow TS^{\ast}x.$ But $TSx_{\alpha}=Tx_{\alpha}\uparrow Tx.$ So
$Tx=TS^{\ast}x\in E^{u}$ and in particular $S^{\ast}x\in L^{1}\left(
T\right)  .$ We deduce that $S^{\ast}$ maps $L^{1}\left(  T\right)  $ into
itself and satisfies $\mathcal{T}S^{\ast}=S^{\ast}.$

Obviously $S^{\ast}e=e$ and it remains only to show that $S^{\ast}$ is a Riesz
homomorphism. Let $x,y\in L^{1}(T)_{+}$. Then there exist two nets
$(x_{\alpha})$ and $(y_{\alpha})$ in $\subset E_{+}$ with $x_{\alpha}\uparrow
x$ and $y_{\alpha}\uparrow y$. Then $x_{\alpha}\vee y_{\alpha}\uparrow x\vee
y$ and so%
\begin{align*}
S^{\ast}(x\vee y)  &  =\sup_{\alpha}S(x_{\alpha}\vee y_{\alpha})=\sup_{\alpha
}(Sx_{\alpha}\vee Sy_{\alpha})\\
&  =\sup_{\alpha}Sx_{\alpha}\vee\sup_{\alpha}Sy_{\alpha}=Sx\vee Sy.
\end{align*}
This completes the proof.
\end{proof}

As is customary, the unique extension of $S$ to $L^{1}\left(  T\right)  $ will
be denoted as $S.$ We can present the analogue of \cite[Theorem 4.4]{b-1662}
within the context of Riesz spaces.

\begin{theorem}
\label{L1}\textit{The conditional expectation preserving system} $(E,T,S,e)$
\textit{is ergodic if and only if} the following condition holds:%
\begin{equation}
L_{S}f=Tf\ \text{for all }f\in L^{1}(T). \label{L1-C}%
\end{equation}

\end{theorem}

\begin{proof}
Suppose first that the system is ergodic and let $f\in L^{1}(T)$. As
$L^{1}(T)$ is $T$-universally complete, we have $SL_{S}f=L_{S}f$ and
$TL_{S}f=Tf$ (see \cite[Theorem 3.9]{L-33}). But since the system is ergodic,
$TL_{S}f=L_{S}f$. It follows that $L_{S}f=Tf$. Conversely assume that
condition (\ref{L1-C}) holds and let $f\in E$ satisfying $Sf=f$. Thus
$Tf=L_{S}f=f$ which shows that the system is ergodic and concludes the proof.
\end{proof}

\section{Some applications}

As application of previous results we will show that $S$ is an isometry from
$L^{p}\left(  T\right)  $ to itself when $L^{p}\left(  T\right)  $ is equipped
with its vector--valued norm $\left\Vert .\right\Vert _{p,T}.$

\begin{theorem}
\label{LL}\textit{Let} $(E,e,T,S)$ \textit{be a conditional expectation
preserving system and let }$q\in\lbrack1,\infty).$\textit{ Then }$L^{q}\left(
T\right)  $ is $S$-invariant. Moreover%
\begin{equation}
\left\Vert Sx\right\Vert _{T,q}=\left\Vert x\right\Vert _{T,q}\text{ for all
}x\in L^{q}\left(  T\right)  . \label{E2}%
\end{equation}

\end{theorem}

\begin{proof}
(i) $q<\infty.$ Assume first that $x$ is a component of $e$, then $Sx$ is also
a component of $e$ and then $Sx=\left(  Sx\right)  ^{q}$ and $x^{q}=x.$ Hence
equality \ref{E2} holds for $x.$. Now, let $x=%
%TCIMACRO{\tsum \limits_{k}}%
%BeginExpansion
{\textstyle\sum\limits_{k}}
%EndExpansion
\alpha_{k}p_{k}$ be a finite linear combination of disjoint components with
$\lambda_{k}\geq0.$ Then $\left(  Sp_{k}\right)  $ are disjoint components of
$e$. Hence%
\[
\left(  Sx\right)  ^{q}=(%
%TCIMACRO{\tsum \limits_{k}}%
%BeginExpansion
{\textstyle\sum\limits_{k}}
%EndExpansion
\lambda_{k}Sp_{k})^{q}=%
%TCIMACRO{\tsum \limits_{k}}%
%BeginExpansion
{\textstyle\sum\limits_{k}}
%EndExpansion
\lambda_{k}^{q}Sp_{k}.
\]
It follows that%
\[
T\left(  Sx\right)  ^{q}=%
%TCIMACRO{\tsum \limits_{k}}%
%BeginExpansion
{\textstyle\sum\limits_{k}}
%EndExpansion
\lambda_{k}^{q}TSp_{k}=%
%TCIMACRO{\tsum \limits_{k}}%
%BeginExpansion
{\textstyle\sum\limits_{k}}
%EndExpansion
\lambda_{k}^{q}Tp_{k}=T\left(
%TCIMACRO{\tsum \limits_{k}}%
%BeginExpansion
{\textstyle\sum\limits_{k}}
%EndExpansion
\lambda_{k}^{q}p_{k}\right)  =Tx^{q}.
\]
Now let $x\in L^{q}\left(  T\right)  .$ We have to show that $T\left\vert
x\right\vert ^{q}=T\left\vert Sx\right\vert ^{q}=T\left(  \left(  S\left\vert
x\right\vert \right)  ^{q}\right)  .$ We can assume without loss of generality
that $x\in L^{q}\left(  T\right)  ^{+}.$ By Freudenthal Theorem there exists a
sequence $\left(  x_{k}\right)  $ of e-step functions such that $x_{k}\uparrow
x.$ We infer from \cite[Theorem 4.6]{L-06} that $x_{k}^{q}\uparrow x^{q}.$
Using order continuity of $S$ and the prior case we conclude that
$TSx^{q}=Tx^{q},$ as per our requirement.

(ii) $q=\infty.$ Let $x\in L^{\infty}\left(  T\right)  .$ Then $x\in
L^{q}\left(  T\right)  $ for all $q\in\left(  1,\infty\right)  .$ If follows
from the first case that $\left\Vert x\right\Vert _{T,q}=\left\Vert
Sx\right\Vert _{q,T}.$ According to \cite[Theorem 3.13]{L-180} we deduce that
$\left\Vert x\right\Vert _{\infty,T}=\left\Vert Sx\right\Vert _{\infty,T}.$
\end{proof}

As a consequence of Theorem \ref{L1} we give multiple characterizations of
ergodicity. The equivalence between (i) and (iii) below is already established
in \cite{L-724}. For the classical case, the reader can consult \cite[Theorem
8.10]{b-2982}.

\begin{theorem}
Let $(E,T,S,e)$ be a conditional expectation preserving system. Then the
following statements are equivalent.

\begin{enumerate}
\item[(i)] The system $(E,T,S,e)$ is ergodic;

\item[(ii)] $\lim\limits_{n\longrightarrow\infty}\dfrac{1}{n}%
%TCIMACRO{\tsum \limits_{k=0}^{n-1}}%
%BeginExpansion
{\textstyle\sum\limits_{k=0}^{n-1}}
%EndExpansion
T\left(  fS^{k}g\right)  =TfTg$ for all $f\in L^{\infty}\left(  T\right)  $
and $g\in L^{1}\left(  T\right)  .$

\item[(iii)] $\lim\limits_{n\longrightarrow\infty}\dfrac{1}{n}%
%TCIMACRO{\tsum \limits_{k=0}^{n-1}}%
%BeginExpansion
{\textstyle\sum\limits_{k=0}^{n-1}}
%EndExpansion
T\left(  fS^{k}g\right)  =TfTg$ for all $f,g\in E_{e}.$

\item[(iv)] $\lim\limits_{n\longrightarrow\infty}\dfrac{1}{n}%
%TCIMACRO{\tsum \limits_{k=0}^{n-1}}%
%BeginExpansion
{\textstyle\sum\limits_{k=0}^{n-1}}
%EndExpansion
T\left(  fS^{k}g\right)  =TfTg$ for all components $f,g$ of $e;$

\item[(v)] $\lim\limits_{n\longrightarrow\infty}\dfrac{1}{n}\sum
\limits_{k=0}^{n-1}T(fS^{k}f)=(Tf)^{2}$ \textit{for all} $f\in E_{e}$.

\item[(vi)] $\lim\limits_{n\longrightarrow\infty}\dfrac{1}{n}%
%TCIMACRO{\tsum \limits_{k=0}^{n-1}}%
%BeginExpansion
{\textstyle\sum\limits_{k=0}^{n-1}}
%EndExpansion
T\left(  fS^{k}f\right)  =\left(  Tf\right)  ^{2}$ for all component $f$ of
$e.$
\end{enumerate}
\end{theorem}

\begin{proof}
(i) $\Rightarrow$ (ii) We have by \cite[Theorem 3.7]{L-180},%
\[
\left\vert T\left(  fS_{n}g\right)  -TfTg\right\vert =\left\vert Tf\left(
S_{n}g-Tg\right)  \right\vert \leq\left\Vert f\right\Vert _{T,\infty
}\left\Vert S_{n}g-Tg\right\Vert _{T,1}%
\]
By Theorem \ref{L1}, $S_{n}g\overset{o}{\longrightarrow}Tg,$ so by order
continuity of $T,\left\Vert S_{n}g-Tg\right\Vert _{T,1}\overset{o}%
{\longrightarrow}0$ and the result follows.

This is done in \cite[Theorem 3.5]{L-724}.

(ii) $\Longrightarrow$ (iii) $\Rightarrow$ (vi) and (v) $\Rightarrow$ (vi) are obvious.

(iv) $\Longrightarrow$ (v) We have to show that $TfS_{n}f\overset
{o}{\longrightarrow}\left(  Tf\right)  ^{2}$ if $f\in E_{e}.$ This is true by
assumption for components. Assume now that $f=\sum_{k}\lambda_{k}p_{k}$ is an
e-step functions, where $\lambda_{k}$ are reals and $p_{k}$ are components. We
have
\[
TfS_{n}f=%
%TCIMACRO{\tsum \limits_{j;k}}%
%BeginExpansion
{\textstyle\sum\limits_{j;k}}
%EndExpansion
\lambda_{k}\lambda_{j}T\left(  p_{j}S_{n}p_{k}\right)  \overset{o}%
{\longrightarrow}%
%TCIMACRO{\tsum \limits_{j;k}}%
%BeginExpansion
{\textstyle\sum\limits_{j;k}}
%EndExpansion
\lambda_{k}\lambda_{j}Tp_{j}Tp_{k}\text{ as }n\longrightarrow\infty.
\]
by (iv), which proves the result as
\[%
%TCIMACRO{\tsum \limits_{j;k}}%
%BeginExpansion
{\textstyle\sum\limits_{j;k}}
%EndExpansion
\lambda_{k}\lambda_{j}Tp_{j}Tp_{k}=\left(
%TCIMACRO{\tsum \limits_{k}}%
%BeginExpansion
{\textstyle\sum\limits_{k}}
%EndExpansion
\lambda_{k}Tp_{j}\right)  ^{2}=\left(  Tf\right)  ^{2}.
\]
Assume now that $f\in E_{e}$ then there exist a sequence of e-steps functions
$\left(  f_{k}\right)  $ and a sequence $\left(  \varepsilon_{n}\right)  $ of
real numbers such that $\varepsilon_{k}\downarrow0$ and $\left\vert
f-f_{k}\right\vert \leq\varepsilon_{k}e.$ The sequence $\left(  f_{k}\right)
$ is order bounded in $E_{e},$ so $\left\vert f_{k}\right\vert \leq Me$ for
some real $M$ and then $T\left\vert f_{k}\right\vert \leq Me$ and $S\left\vert
f_{k}\right\vert \leq Me.$ We deduce that $\left\vert Sf\right\vert \leq Me,$
$T\left\vert f\right\vert \leq Me$ and%
\[
\left\vert \left(  Tf_{k}\right)  ^{2}-\left(  Tf\right)  ^{2}\right\vert
=\left\vert Tf_{k}-Tf\right\vert \left\vert Tf_{k}+Tf\right\vert
\leq2\varepsilon_{k}Me.
\]
Now, it follows from the decomposition%
\begin{align*}
T\left(  fS_{n}f\right)  -\left(  Tf\right)  ^{2}  &  =T\left[  \left(
f-f_{k}\right)  S_{n}f\right]  +T\left[  f_{k}S_{n}\left(  f-f_{k}\right)
\right] \\
&  +T\left[  f_{k}S_{n}f_{k}\right]  -\left(  Tf_{k}\right)  ^{2}+\left(
Tf_{k}\right)  ^{2}-\left(  Tf\right)  ^{2}.
\end{align*}
that
\begin{align*}
\left\vert TfSf-\left(  Tf\right)  ^{2}\right\vert  &  \leq2\varepsilon
_{k}Me+\left\vert T\left[  f_{k}S_{n}f_{k}\right]  -\left(  Tf_{k}\right)
^{2}\right\vert +2\varepsilon_{k}Me\\
&  =4\varepsilon_{k}Me+\left\vert T\left[  f_{k}S_{n}f_{k}\right]  -\left(
Tf_{k}\right)  ^{2}\right\vert .
\end{align*}
Taking the $\lim\sup$ over $n$ we get%
\[
\limsup\limits_{n}\left\vert T\left(  fS_{n}f\right)  -\left(  Tf\right)
^{2}\right\vert \leq4\varepsilon_{k}Me.
\]
As this happens for every $k$ we get $\limsup\limits_{n}\left\vert
TfSf-\left(  Tf\right)  ^{2}\right\vert =0$ and then $TfS_{n}f\overset
{o}{\longrightarrow}\left(  Tf\right)  ^{2}$ as required.

(vi) $\Longrightarrow$ (i) Let $p$ be a component such that $Sp=p.$ Then
$S_{n}p=p$ so (v) gives $Tp=\left(  Tp\right)  ^{2},$ that is $Tp$ is a
component. By Lemma \ref{TpC} $Tp=p.$ According to Lemma \ref{LC} the system
is ergodic. vv
\end{proof}

\begin{center}

\end{center}

\end{document}